\documentclass[preprint,12pt]{elsarticle}

\usepackage{amsmath, amsthm, amssymb, bm, graphicx, mathrsfs}
\usepackage{color}
\numberwithin{equation}{section}
\newtheorem{Th}{Theorem}[section]
\newtheorem{Le}{Lemma}[section]
\newtheorem{Co}{Condition}[section]

\journal{CSIAM Transactions on Applied Mathematics}
\begin{document}

\begin{frontmatter}



\title{ { A priori} bounds for elastic scattering by deterministic and random unbounded rough surfaces}


\author[inst1]{Tianjiao Wang}

\affiliation[inst1]{organization={School of Mathematical Sciences},
	addressline={Zhejiang University}, 
	city={Hangzhou},
	postcode={310058}, 
	country={P. R. China}}

\author[inst2]{Yiwen Lin}
\affiliation[inst2]{organization={School of Mathematical Sciences and Institute of Natural	Sciences},
	addressline={Shanghai Jiao Tong University}, 
	city={Shanghai},
	postcode={200240}, 
	country={P. R. China}}

\author[inst1,inst3]{Xiang Xu}

\fntext[inst3]{Corresponding author, xxu@zju.edu.cn. This work was supported in part by National Natural Science Foundation of China (11621101, 12071430, 12201404), Key Laboratory of Collaborative Sensing and Autonomous Unmanned Systems of Zhejiang Province, and Postdoctoral Science Foundation of China (2021TQ0203).}
\begin{abstract}
This paper investigates the elastic scattering by unbounded deterministic and random rough surfaces, which both are assumed to be graphs of Lipschitz continuous functions. For the deterministic case,  an \textit{a priori} bound  explicitly dependent on frequencies is derived by the variational approach. For the scattering by random rough surfaces with a random source, well-posedness of the corresponding variation problem is proved. Moreover, a similar bound with explicit dependence on frequencies for the random case is also established based upon the deterministic result, Pettis measurability theorem and Bochner’s integrability Theorem.
\end{abstract}


\end{frontmatter}


\section{Introduction}
This paper considers mathematical analysis of time-harmonic elastic waves scattered by unbounded deterministic and random rough surfaces in two-dimensions. Elastic scattering problems have received intensive attentions both in mathematics and engineering because of their wide-ranging applications in seismology and geophysics (see \cite{r1,r2,r3}). Mathematically, elastic wave scattering can be formulated as a boundary value problem of the Naiver equation which is more complicated than electromagnetic and acoustic equations.   

Considerable efforts have been devoted to electromagnetic and acoustic rough surface scattering. For instance, Chandler-Wilde and Zhang proposed an upward radiation condition (UPRC) of the Helmholtz equation and studied the Green function and potentials of electromagnetic scattering by rough surfaces in \cite{r5}. Furthermore, they employed an integral equation method to prove the corresponding existence and uniqueness in \cite{r6}. Moreover, variation approaches are utilized to prove the well-posedness based on Rellich identities which imply an \textit{a priori} bound with explicit dependence to the wave number in \cite{r7}. Recently, Chandler-Wilder and Elschner extended the well-posedness in weighted Sobolev spaces by variation approaches and used the finite element method with perfectly matched layer technique to solve acoustic scattering by rough surfaces in \cite{r9}. For the scattering with tapered incident wave by fractal rough surface, Zhang, Ma and Wang used regularized conjugate gradient method to reconstruct the surface in \cite{r26}. Zhang, Wang, Feng and Li \cite{r24} obtained the Fr$\acute{\rm e}$chet derivative of the scattered field which can be used to give numerical methods for shape reconstruction from multi-angle and multi-frequency data. Similar results for general unbound rough surface was given by Zhang and Ma in \cite{r25}. Bao and Zhang realized the reconstruction from multi-frequency phaseless data in \cite{r22} and obtained the uniqueness and existence for direct problem and uniqueness for inverse problem based on boundary integral equations in \cite{r21}. Numerical method for recovering localized perturbation of unbounded surface via near-field  is proposed in \cite{r23} by Bao and Lin.

Compared to electromagnetic and acoustic scattering, results on elastic scattering from unbounded rough surfaces are relatively fewer. Arens investigated the Green tensor, elastic potentials, UPRC and proved uniqueness and existence by integral equation methods in \cite{r10,r11,r12}. Elschner and Hu deduced a transparent boundary condition and proved existence and uniqueness by variation approaches based on the Rellich identity in \cite{r13}. Furthermore, they studied the solvablity in weighted Sobolev spaces, on which they based to prove the existence and uniqueness of elastic scattering by unbounded rough surfaces with a plane or point source incident wave in \cite{r14}. Recently Hu, Li and Zhao generalized the similar results for three-dimensions in \cite{r16}.

For random cases, Warnick and Chew \cite{r8} proposed a numerical method to solve electromagnetic scattering from random rough surfaces. Pembery and Spence \cite{r17} considered the Helmholtz equation in random media and proposed a general framework to study the variation problem, which overcomes the difficulties on both lacks of coercivity and the necessary compactness in Bochner's spaces. Bao, Lin and Xu 
\cite{r18} extended this general framework to obtain an explicit stability result with respect to the wave number for electromagnetic scattering by random periodic surfaces.

In this paper, we derive an \textit{a priori} bound explicitly dependent on the frequency and the measured height for the deterministic elastic scattering by rough surfaces based on Rellich identities. Different from electromagnetic scattering, direct applying Rellich identities is not enough for elastic scattering. By the method in \cite{r13}, we use the \textit{a priori} bound for Helmholtz equations and construct a boundary value problem of a Helmholtz equation to overcome the difficulty.  Moreover, for the random case, we prove the well-posedness of the stochastic variation problem and extends the explicit bound based on the framework in \cite{r17}. The main difference with \cite{r17} is that the variation forms for different samples are defined in different Banach spaces. So we need to use the method of changing variables proposed by Kirsch in \cite{r19} to transform the variation formulas into a deterministic domain but with random medium. And for any given sample, the transformed variation problem would be of the same well-posedness with the original variation problem suppose that we choose a sufficient large measured height such that the transform is invertible. Compared with \cite{r18}, the main difference is the inhomogeneous source term is also random, so we construct a product topology space be the image space of the input map and consider the continuity in the product topology.

The paper is outlined as follows. In Section 2, formulations of deterministic and random rough surfaces scattering are introduced and two corresponding variation problems are proposed respectively. Section 3 is devoted to derive an \textit{a priori} bound with explicit dependence on frequencies and measured height. In Section 4, the well-posedness of random variation problem is derived. Finally, conclusions are given in Section 5. Without additional explanation, $C$ is a constant independent on the frequency $\omega$, the measured height $h$ and Lipschitz constant $L$ in Section 3 and independent on random sampling $\eta$ in Section 2 and Section 4.

\section{Problem formulation}

This section introduces mathematical formulations of deterministic and random elastic scattering by rough surfaces.

\subsection{Deterministic problem}
 As shown in Figure 1, assume $D \subset \mathbb{R}^2$ is an unbounded connected open set in the upper half space. The curve $\partial D=S$ is assumed to be the graph of a Lipschitz continuous function with Lipschitz constant $L$, i.e.,\[
S=\{x \in \mathbb{R}^2: x_2=f(x_1),x_1\in \mathbb{R}\},\]where\[|f(s)-f(t)| \le L |s-t| \quad \forall s,t \in \mathbb{R}.
\]
\begin{figure}
    \centering
    \includegraphics[scale=0.8]{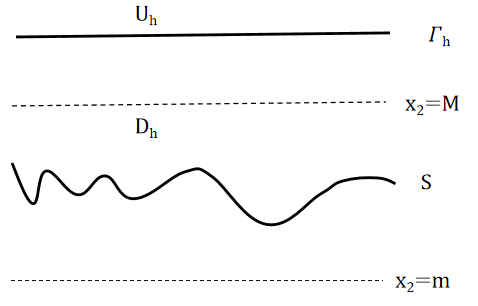}
    \caption{The problem geometry}
    \label{fig:my_label}
\end{figure}
In this paper, the function $f$ is assumed to satisfy $
m<f<M$ with constants $m,M \in \mathbb{R}$.
For $h>M$, denote $\Gamma_h=\{x \in \mathbb{R}^2: x_2=h\}$ and $U_h=\{x \in \mathbb{R}^2: x_2>h\}$. Then $D_h$ is defined by $D_h=D\backslash \bar{U}_h$. Assume the
 inhomogeneous source term $g \in L^2(D)^2$. Its support is assumed to be in $D_h$ in this paper. The elastic wave satisfies the inhomogeneous Navier equations, i.e,\begin{equation*}
     \mu\Delta u+(\mu+\lambda)\nabla(\nabla\cdot u)+\omega^2u=g\quad \text{in}\quad D,
\end{equation*}
where Lam$\acute{\rm e}$ constants $\lambda>0$, $\mu>0$ and frequency $\omega>0$.
For convenience, let \[
\Delta^*u=\mu\Delta u+(\mu+\lambda)\nabla(\nabla\cdot u).
\]
Moreover, throughout this paper, we consider the Dirichlet boundary condition\begin{equation*}
    u=0\quad \text{on}\quad S.
\end{equation*} Next we briefly introduce the transparent boundary condition to reduce the unbounded problem to be bounded, where the details can be found in \cite{r13}.
We begin by the Helmholtz decomposition for $u$:
\begin{equation}  \label{eq2.1}
    u=\frac{1}{i}({\rm grad}\,\phi+\overrightarrow{{\rm curl}}\,\psi)
    \end{equation}
    with
    \begin{equation}  \label{eq2.2}
    \quad \phi :=-\frac{i}{k^2_p}{\rm div}\,u,\psi:=\frac{i}{k^2_s}{\rm curl}\,u,
\end{equation}
where $\overrightarrow{{\rm curl}}=(\partial_2 , -\partial_1 )^\top$, ${\rm curl~} u = \partial_1 u_2 - \partial_2 u_1$.
The scalar functions $\phi$ and $\psi$ satisfy the homogeneous Helmholtz equations
 \begin{equation}  \label{eq2.3}
     (\Delta+k^2_p)\phi=0\quad {\rm and}\quad (\Delta+k^2_s)\psi=0,\quad \text{in}\quad U_h.
 \end{equation}
 The Fourier transform of $\phi$ and $\psi$ has the form
\begin{equation}  \label{eq2.4}
     \hat{\phi}=P_h(\xi)\,\text{exp}(i (x_2-h)\gamma_p(\xi)),\quad
     \hat{\psi}=S_h(\xi)\,\text{exp}(i (x_2-h)\gamma_s(\xi)),
 \end{equation}
  where \[
\gamma_p(\xi)=\sqrt{k^2_p-\xi^2},\, \gamma_s(\xi)=\sqrt{k^2_s-\xi^2}.
\] and $\hat{u}=\mathcal{F}u$ is the Fourier transform of $u$ with respect to $x_1$. Here $P_h(\xi),S_h(\xi) \in L^2(\mathbb{R})$ can be represented by 
    \begin{equation}  \label{eq2.5}
    \left(\begin{array}{cc}
         P_h(\xi)\\
         S_h(\xi)
    \end{array}\right)=\frac{1}{\xi^2+\gamma_p\gamma_s}\left(
    \begin{array}{cc}
         \xi &\gamma_s  \\
         \gamma_p& -\xi
    \end{array}\right)\left(\begin{array}{cc}
    \hat{u}_{s,1}(\xi,h)\\
    \hat{u}_{s,2}(\xi,h)\\
    \end{array}\right).
\end{equation}
 The function $u$ is required to satisfy the upward radiation condition 
 \begin{equation}  \label{eq2.6}
    u=\frac{1}{\sqrt{2\pi}}\int_\mathbb{R}\big( \text{exp}(i x_2\gamma_p(\xi))M_p(\xi)+\text{exp}(i x_2\gamma_s(\xi))M_s(\xi)\big) \hat{u}(\xi,h)\text{exp}(ix\xi)\,\mathrm{d}\xi
\end{equation}
in $U_h$ with
\[
 M_p(\xi)=\frac{1}{\xi^2+\gamma_p\gamma_s}\left(
    \begin{array}{cc}
         \xi^2 &\gamma_s\xi  \\
         \gamma_p\xi& \gamma_p\gamma_s
    \end{array}\right),\, \,
     M_s(\xi)=\frac{1}{\xi^2+\gamma_p\gamma_s}\left(
    \begin{array}{cc}
         \gamma_p\gamma_s &-\gamma_s\xi  \\
         -\gamma_p\xi& \xi^2
    \end{array}\right).
\]
Define a differential operator $T$ by
 \begin{equation}\label{eq2.7}
    Tu:=\mu\partial_nu+(\lambda+\mu)\Vec{n} {\rm div}\,u \quad {\rm on}\quad\Gamma_h.
\end{equation}
Combining \eqref{eq2.6}-\eqref{eq2.7} gives
\begin{equation*}
    Tu=\frac{1}{\sqrt{2\pi}}\int_\mathbb{R}M(\xi)\hat{u}(\xi,h)\text{exp}(ix\xi)\,\mathrm{d}\xi,
\end{equation*}
where \begin{equation} \label{eq2.8}
  M(\xi)=  \frac{i}{\xi^2+\gamma_p\gamma_s}\left(
    \begin{array}{cc}
         \omega^2\gamma_p &-\xi\omega^2+\xi\mu(\xi^2+\gamma_p\gamma_s)  \\
         \xi\omega^2-\xi\mu(\xi^2+\gamma_p\gamma_s)& \omega^2\gamma_s
    \end{array}\right).
\end{equation}Then the Dirichlet to Neumann (DtN) operator $\mathcal{T}$ can be defined by \begin{equation*}
\mathcal{T}f:=\mathcal{F}^{-1}(M\hat{f}),\quad f\in H^{1/2}(\mathbb{R}).
\end{equation*}
Therefore, the transparent boundary condition can be given by
\[
Tu=\mathcal{T}u \quad \text{on}\quad \{x_2=h\}.
\] Furthermore, according to the above TBC, the original scattering problem in $D$ can be reduced into $D_h$:
\begin{align*}
	\begin{array}{rll}
\Delta^*u+\omega^2u=g & {\rm in} &D_h, \\
u=0& {\rm on} &S, \\
Tu=\mathcal{T}u & {\rm on}&\Gamma_h.
\end{array}
\end{align*}
In order to investigate the variation formulation of this reduced problem, we introduce a function space \[V_h(D_h):= \{ u \in H^1(D_h)^2: u=0\, {\rm ~on~}\, S \}.\]For convenience, denote $V_h=V_h(D_h)$.
Suppose $u,\,v \in V_h$, the Betti formula gives
\begin{equation*}
    -\int_{D_h}g\cdot \bar{v}\,\text{d}x=-\int_{D_h} (\Delta^*+\omega^2)u\cdot\bar{v}\,\mathrm{d}x= \int_{D_h} \mathcal{E}(u,\bar{v})-\omega^2 u\cdot\bar{v}\,\mathrm{d}x-\int_{\Gamma_h}\mathcal{T}u\cdot\bar{v}\,\mathrm{d}s,
\end{equation*}
where\begin{equation*}
    \mathcal{E}(u,v)=\mu(\nabla u_1\cdot\nabla v_1+\nabla u_2\cdot \nabla v_2)+(\lambda+\mu)(\nabla\cdot u)(\nabla\cdot v).
\end{equation*}
Define the sesquilinear form $B\, :V_h\times V_h\to \mathbb{C}$ by 
\begin{equation*}
    B(u,v)=
\int_{D_h} \mathcal{E}(u,\bar{v})-\omega^2 u\cdot\bar{v}\,\mathrm{d}x-\int_{\Gamma_h}\mathcal{T}u\cdot\bar{v}\,\mathrm{d}s.
\end{equation*}
Now we can give the variation formula for deterministic problem.

\emph{Variation problem 1} (VP 1): Find $u\in V_h$ such that \[B(u,v)=-(g,v)_{D_h} ,\,\forall  v\in V_h.\]
\subsection{Random problem}
Let  $(\Omega,\mathcal{A},\mathbb{P})$ be a complete probability space. Denote by $S(\eta)$ a random surface 
\[
S(\eta):=\{x \in \mathbb{R}^2 :x_2=f(\eta;x_1),\eta \in \Omega, x_1 \in \mathbb{R}\}.
\] Similarly, $D(\eta)$ and $D_h(\eta)$ represent the random counterparts of $D$ and $D_h$, respectively. Assume $f(\eta;x_1)$ is a Lipschitz continuous function with Lipschitz constant $L(\eta)$ for all $\eta \in \Omega$ and it also satisfies $m<f(\eta; x_1)<M$.
The random inhomogeneous source $g(\eta)$ is assumed to satisfy $g(\eta) \in L^2(D(\eta))^2$ with its support in $D_h(\eta)$. Similarly as the deterministic case, we can give the following random boundary value problem.
\begin{align*}
	\begin{array}{rll}
\Delta^*u(\eta;\cdot)+\omega^2u(\eta;\cdot)=g(\eta;\cdot)&{\rm in}&D_h(\eta), \\
u(\eta;\cdot)=0&{\rm on}&S(\eta), \\
Tu(\eta;\cdot)=\mathcal{T}u(\eta;\cdot)&{\rm on}&\Gamma_h.
\end{array}
\end{align*}
For simplicity, let $V_h(\eta)=V_h(D_h(\eta))$.
Define a sesquilinear form $\tilde{B}_\eta$ on $V_h(\eta)\times V_h(\eta)$ by \begin{equation}  \label{eq2.9}
    \tilde{B}_\eta(u,v)=
\int_{D_h(\eta)} \mathcal{E}(u,\bar{v})-\omega^2 u\cdot\bar{v}\,\mathrm{d}x-\int_{\Gamma_h}\mathcal{T}u\cdot\bar{v}\,\mathrm{d}s,
\end{equation}
and an antilinear functional $\tilde{G}_\eta$ on $V_h(\eta)$ by
\begin{equation}  \label{eq2.10}
    \tilde{G}_\eta(v):=-\int_{D_h(\eta)}g(\eta)\cdot\bar{v}\,\text{d}x.
\end{equation} Then we want to define the stochastic variation problem. Direct definition is not allowed because $V_h(\eta)$ is dependent on $\eta$. By the method in \cite{r19}, variable transform can give a new sesquilinear form defined on $V_h \times V_h$. This implies that we can define stochastic variation problem after variable transform. Let $f_0=f(\eta_0)$ and $g_0=g(\eta_0)$ for some fixed $\eta_0 \in \Omega$. Then let $D=D(\eta_0)$, $D_h=D_h(\eta_0)$  and $ V_h=V_h(\eta_0)$ for convenience.

In addition, we assume $g(\eta) \in H^1(D(\eta))^2$
and $f(\eta)$ is assumed to satisfy \[
\|f(\eta)-f_0\|_{1,\infty} \le M_0, \quad \forall \eta \in \Omega,
\] with constant $M_0 >0$. The measured height $h$ is chosen such that\begin{equation}\label{eq2.11}
    (M-m)/\gamma<1,
\end{equation} where $\gamma=h-\sup\limits_{x_1}f_0(x_1)$. 

Denote by $Lip (\mathbb{R})$ the set including all Lipschitz continuous functions on $\mathbb{R}$. Then define a product topology space \[
\mathcal{C}= \mathcal{C}_1 \times \mathcal{C}_2, 
\]where\[
\mathcal{C}_1:=\{v \in Lip(\mathbb{R}) :m<v<M, \|v-f_0\|_{1,\infty}\le M_0\},
\] with constant $M_0>0$
and\[
\mathcal{C}_2:=H^1_0(D_h)^2.
\]The topology of $ \mathcal{C}_1$ and $\mathcal{C}_2$ is respectively given by norm $\|\cdot\|_{1,\infty}$ and $\|\cdot\|_{H^1(D_h)^2}$. 

Consider the transform $\mathcal{H}$: $D_h \to D_h(\eta)$ defined by \begin{equation*}
    \mathcal{H}(y)=y_2+\alpha(y_2-f_0(y_1))(f(\eta; y_1)-f_0(y_1))e_2,\quad y \in D_h,
\end{equation*}
where $e_2$ is the unit vector in $x_2$ direction and $\alpha(x)$ is a cutoff function which satisfies\[
\alpha(x)=\left\{\begin{array}{cc}
    0, &  x<\delta, \\
    1, &  x>\gamma,
\end{array}\right.
\]with sufficiently small $\delta$.
It is also required to satisfy\begin{equation}
    \label{eq2.12}|\alpha'|<1/(\gamma-2\delta).
\end{equation}
The Jacobi matrix of $\mathcal{H}$ is\begin{equation*}
    \mathcal{J}_\mathcal{H}=I_2+\left(\begin{array}{cc}
        0 & 0 \\
        J_1 & J_2
    \end{array}\right),
\end{equation*}
where \begin{align*}
    J_1&=\alpha(y_2-f_0(y_1))(f'(\eta;y_1)-f'_0(y_1))-\alpha'(y_2-f_0(y_1))f'_0(y_1)(f(\eta;y_1)-f_0(y_1)),\\
    J_2&=\alpha'(y_2-f_0(y_1))(f(\eta;y_1)-f_0(y_1)).
\end{align*}
Since matrix $\mathcal{J}_\mathcal{H}$ is required to be non-singular so that $\mathcal{H}$ is invertible, according to \eqref{eq2.12},  we obtain \[
|J_2|<\frac{M-m}{\gamma-2\delta}.
\] Hence, by \eqref{eq2.11}, we can choose $\delta$ sufficiently small such that \begin{equation}\label{eq2.13}
|J_2|<\frac{M-m}{\gamma-2\delta}<1,
\end{equation} which implies that $\mathcal{H}$ is invertible.
It is easy to verify $\mathcal{H}(\Gamma_h)=\Gamma_h$. For $u,v \in V_h(\eta)$, taking $x=\mathcal{H}(y)$ in \eqref{eq2.9} yields
 \begin{align*}
    \tilde{B}_\eta(u,v)=&\mu\int_{D_h} \sum_{j=1}^2\nabla\tilde{u}_j\mathcal{J}_{\mathcal{H}^{-1}}\mathcal{J}_{\mathcal{H}^{-1}}^\top\nabla \bar{\tilde{v}}_j\det{\mathcal{J}_\mathcal{H}}\,\text{d}y \\
    &+(\lambda+\mu)\int_{D_h} (\nabla\tilde{u}:\mathcal{J}_{\mathcal{H}^{-1}})(\nabla\bar{\tilde{v}}:\mathcal{J}_{\mathcal{H}^{-1}}^\top)\det{\mathcal{J}_\mathcal{H}}\,\text{d}y \\
    &-\omega^2\int_{D_h}\tilde{u}\cdot\bar{\tilde{v}}\det{\mathcal{J}_{\mathcal{H}}}\,\text{d}y 
    -\int_{\Gamma_h}\mathcal{T}\tilde{u}\cdot\bar{\tilde{v}} \,\text{d}s(y),
\end{align*}
 where $\tilde{u}=u\circ\mathcal{H}$, $\tilde{v}=v\circ\mathcal{H}$ and \[
A:B=\text{tr}(B^\top A) \quad A,B \in \mathbb{C}^{2\times2}.
 \]
 Similarly, for $v \in V_h(\eta)$, let $x=\mathcal{H}(y)$ in \eqref{eq2.10}, \begin{equation*}
     \tilde{G}_\eta(v)=-\int_{D_h}\tilde{g}(\eta)\cdot\bar{\tilde{v}}\det{\mathcal{J}_\mathcal{H}}\,\text{d}x.
 \end{equation*} 
 Recall that we require $g(\eta) \in H^1(D(\eta))^2$ and the support of $g(\eta)$ is in $D_h(\eta)$, we have $\tilde{g}(\eta) \in H^1_0(D_h)^2$ for all $\eta$. So we can define the input map $c$ : $\Omega\to\mathcal{C}$ by \[
c(\eta):=(f(\eta),\tilde{g}(\eta)).
\]
  Note that $\tilde{u},\tilde{v} \in V_h$. Thus we can define a continuous sesquilinear form $B_{c(\eta)}(u,v)$ on $V_h \times V_h$ by \begin{align}\label{eq2.14}
B_{c(\eta)}(u,v):=&\mu\int_{D_h} \sum_{j=1}^2\nabla{u}_j\mathcal{J}_{\mathcal{H}^{-1}}\mathcal{J}_{\mathcal{H}^{-1}}^\top\nabla\bar{{v}}_j\det{\mathcal{J}_\mathcal{H}}\,\text{d}y \notag \\
    &+(\lambda+\mu)\int_{D_h} (\nabla{u}:\mathcal{J}_{\mathcal{H}^{-1}})(\nabla\bar{{v}}:\mathcal{J}_{\mathcal{H}^{-1}}^\top)\det{\mathcal{J}_\mathcal{H}}\,\text{d}y \notag \\
    &-\omega^2\int_{D_h}{u}\cdot\bar{{v}}\det{\mathcal{J}_\mathcal{H}}\,\text{d}y 
    -\int_{\Gamma_h}\mathcal{T}{u}\cdot\bar{{v}} \,\text{d}s(y).
  \end{align}
  It is easy to see\begin{equation}\label{eq2.15}
      \tilde{B}_\eta(u,v)=B_{c(\eta)}(\tilde{u},\tilde{v}).
  \end{equation}
 Similarly we can define an antilinear functional $G_{c(\eta)}$ on $V_h$ by \begin{equation}\label{eq2.16}
     G_{c(\eta)}(v):=-\int_{D_h}\tilde{g}(\eta)\cdot\bar{{v}}\det{\mathcal{J}_\mathcal{H}}\,\text{d}x.
 \end{equation}
Obviously, the identity \begin{equation}\label{eq2.17}
     G_{c(\eta)}(\tilde{v})=\tilde{G}_\eta(v)
\end{equation} holds.
 
 Then the sesquilinear form $\tilde{\mathcal{B}}$ on $L^2(\Omega; V_h)\times L^2(\Omega; V_h)$ can be defined by\begin{equation*}
   \mathcal{B}(u,v):=\int_\Omega B_{c(\eta)}(u,v)\,\text{d}\mathbb{P}(\eta).
\end{equation*}
and the antilinear functional $\mathcal{G}$ is defined on $L^2(\Omega;V_h)$ by\begin{equation*}
    \mathcal{G}(v):=\int_\Omega G_{c(\eta)}(v)\,\text{d}\mathbb{P}(\eta).
\end{equation*}
For convenience, we regard sesquilinear form $B_{c(\eta)}$ : $V_h \times V_h \to \mathbb{C}$ as the same operator in $B(V_h, V_h^*)$ generated by it. Here $V_h^*$ is the dual space of $V_h$ and $ B(X,Y)$ denote the space including all bounded linear operators $X \to Y$. Similarly to \eqref{eq2.14} and \eqref{eq2.16}, we can define the sesquilinear form $B_{(\phi,\psi)}$ and the antilinear functional $G_{(\phi,\psi)}$ for all $(\phi,\psi) \in \mathcal{C}$. Then we can define the map $\mathscr{B}$: $\mathcal{C} \to B(V_h,V_h^*) $ by\[
\mathscr{B}((\phi,\psi)):=B_{(\phi,\psi)}
\]
and the map $\mathscr{G}$ : $\mathcal{C} \to  V_h^* $ by
\[
\mathscr{G}((\phi,\psi)):=G_{(\phi,\psi)}.
\]
Now we can define the stochastic variation problem as follows. \par
\emph{Variation problem 2} (VP 2): Find $u \in L^2(\Omega;V_h)$ such that \begin{equation*}
    \mathcal{B}(u,v)=\mathcal{G}(v), \quad \forall v\in L^2(\Omega;V_h).
\end{equation*}

The two variation problems are considered respectively in the following two sections. 
\section{An \textit{a priori} bound for deterministic case}

This section will give an \textit{a priori} bound explicitly dependent on $\omega$, $h$ and $L$. 
Because the matrix $M(\xi)$ is the symbol of the DtN operator, we firstly consider its properties given by the following lemma which shows that the DtN operator is continuous, the real part of $M$ is negative definite when $|\xi|>k_s$ and $M$ is Lipschitz continuous with respect to $\omega$ when $|\xi| \le k_s$. 
\begin{Le}
    (i) For $\xi, \omega \in \mathbb{R}$, $\|M(\xi)\| \le C(\omega) (1+\xi^2)$ and hence the DtN operator $\mathcal{T}$ is continuous. The constant $C(\omega)>0$ is dependent on $\omega$ but independent on $\xi$.
    (ii) For $|\xi|>k_s$ and $\omega \in \mathbb{R}$, $-\Re M(\xi)>0$.
    (iii) For $|\xi|\le k_s$ and $\omega \in \mathbb{R}$, $\|M(\xi)\| \le C\omega$.
\end{Le}
Here $\Re M:=(M+\bar{M}^\top)/2$ and norm $\|\cdot\|$ is defined by $\|A\|:=\max\limits_{i,j}|a_{ij}|$. See Lemma 2 in \cite{r13} for the proof of (i) and (ii). We only prove (iii).
\begin{proof}
Let $\rho=\xi^2+\gamma_p\gamma_s$.
     For $|\xi|\le k_p$, it is easy to see\[
k_p^2\le \rho \le k_pk_s.
     \]
     So we have \begin{equation} \label{eq3.1}
|\omega^2\gamma_p|/\rho\le \omega^2 k_p/k_p^2\le C\omega,
     \end{equation}
     \begin{equation} \label{eq3.2}
     |\omega^2\gamma_s|/\rho\le \omega^2 k_s/k_p^2\le C\omega,
     \end{equation}
     and
     \begin{equation} \label{eq3.3}
|\xi\omega^2-\xi\mu\rho|/\rho \le \omega^2 k_p/k_p^2+\mu k_p \le C\omega.
     \end{equation}
     Combining \eqref{eq3.1}-\eqref{eq3.3} implies
     \[
\|M(\xi)\|\le C\omega,\quad |\xi|\le k_p.
     \]
     For $k_p<|\xi|\le k_s$, we have $k_p^2<|\rho|\le k_s^2$. So it is similar to get \[
\|M(\xi)\|\le C\omega,\quad k_p<|\xi|\le k_s,
     \] which completes the proof.
\end{proof}
Next we give another lemma which can be proved straightly by combining (2.1), (2.5), (2.7)-(2.8) and the variation formula (see \cite{r13}).
\begin{Le}
For the solution $u\in {V}_h \cap H^2(D_h)^2$ to Variation problem 1, the inequality
\begin{equation*}
    \int_{\Gamma_h}\{2\Re (\mathcal{T}u\cdot \partial_2\bar{u})-\mathcal{E}(u,\bar{u})+\omega^2|u|^2\}\,\text{d}s \le 2k_s \Im \int_{D_h} g\cdot \bar{u}\,\text{d}x
    \end{equation*} holds.
\end{Le} 
Now we proceed to prove the \textit{a priori} bound. The strategy is to utilize Rellich identity to estimate $\text{div}\,u$ and $\text{curl}\,u$ on $S$  under the assumption that $g$ and $f$ have sufficient regularities. 
\begin{Le}
    Suppose that $g\in H^1(D)^2$, $f\in C^2(\mathbb{R})$ and $u$ is solution to Variation problem 1. Denote the constant by \[ C_1(L,\omega,h):=C (1+L^2)^{1/2}(\omega(h-m)+1).
    \] Then the inequality\begin{equation*}
\|{\rm div}u\|^2_{L^2(S)}+\|{\rm curl}u\|^2_{L^2(S)} \le C_1(L,\omega,h)\|g\|_{L^2(D_h)^2}\|\partial_2 u\|_{L^2(D_h)^2}
    \end{equation*} holds.
\end{Le}
\begin{proof}
    Since $g\in H^1(D)^2$ and $f\in C^2(\mathbb{R})$, by standard elliptic regularity (see \cite{r20}) we have $u \in H^2(D_h)^2$.
    So multiplying the Navier equations by $\partial_2 \bar{u}$ and integration by parts gives
    \begin{equation} \label{eq3.4}
        2\Re \int_{D_h}\partial_2 \bar{u}\cdot (\Delta^*+\omega^2)u \,\text{d}x=\left(\int_{\Gamma_h}+\int_{S}\right)\{2\Re (Tu\cdot \partial_2 \bar{u})-n_2\mathcal{E}(u,\bar{u})+n_2\omega^2|u|^2\}\,\text{d}s,
    \end{equation} where $n=(n_1,n_2)^\top$ is the unit outward normal vector on $S$.
    In fact, since $D_h$ is an unbounded domain, direct integration by parts is not allowed. Noting $C^\infty_0(D_h \cup \Gamma_h \cup S)^2$ is dense in $H^2(D_h)^2$, we have a sequence $\{u_n\} \subset C^\infty_0(D_h \cup \Gamma_h \cup S)^2$ such that \[
u_n \to u,\quad \text{in}\quad H^2(D_h)^2.
    \]So we firstly use integration by parts to give \eqref{eq3.4} for $u_n$ and then take limits to give the conclusion for $u$.
    
    Note $u=0$ on $S$, which implies $\partial_\tau u=n_1\partial_2 u-n_2\partial_1 u =0$. Inserting it to \eqref{eq3.4} gives \begin{align} \label{eq3.5}
-&\int_{S}\{n_2\mu|\partial_n u|^2+n_2(\lambda+\mu)|\nabla\cdot u|^2\}\,\text{d}s \notag \\
&=\int_{\Gamma_h}\{ 2\Re (Tu\cdot \partial_2\bar{u})-\mathcal{E}(u,\bar{u})+\omega^2|u|^2\}-2\Re \int_{D_h} g\cdot \partial_2\bar{u} \,\text{d}x.
    \end{align}
By Lemma 3.2, it is easy to obtain
\begin{align} \label{eq3.6}
    -&\int_{S}\{n_2\mu|\partial_n u|^2+n_2(\lambda+\mu)|\nabla\cdot u|^2\}\,\text{d}s\notag \\ &\le 2k_s \Im \int_{D_h}g\cdot \bar{u} \,\text{d}x-2\Re\int_{D_h}g\cdot \partial_2\bar{u} \,\text{d}x.
\end{align}
Since
\begin{equation} \label{eq3.7}
    n_2=-(1+f')^{-1/2}\le -(1+L^2)^{-1/2}, 
\end{equation}
combining \eqref{eq3.5}-\eqref{eq3.7} gives \begin{equation} \label{eq3.8}
    \|\text{div}\, u\|^2_{L^2(S)}+\|\partial_n u\|^2_{L^2(S)} \le 2(1+L^2)^{1/2} \left(k_s \Im \int_{D_h}g\cdot \bar{u} \,\text{d}x-\Re\int_{D_h}g\cdot \partial_2\bar{u} \,\text{d}x\right).
\end{equation}
By the Poincar$\acute{\rm e}$ inequality (see Lemma 3.4 in \cite{r7})
\begin{equation} \label{eq3.9}
\|u\|_{L^2(D_h)^2} \le (h-m)/\sqrt{2} \|\partial_2 u\|_{L^2(D_h)^2},
\end{equation} we get \begin{align} \label{eq3.10}
    k_s \Im \int_{D_h}g\cdot \bar{u} \,\text{d}x&-\Re\int_{D_h}g\cdot \partial_2 \bar{u} \,\text{d}x \notag \\ &\le C(\omega(h-m)+1)\|g\|_{L^2(D_h)^2}\|\partial_2 u\|_{L^2(D_h)^2}.
\end{align}
By \eqref{eq3.8}-\eqref{eq3.10}, \begin{equation*}
     \|\text{div} u\|^2_{L^2(S)}+\|\partial_n u\|^2_{L^2(S)} \le C_1(L,\omega,h)\|g\|_{L^2(D_h)^2}\|\partial_2 u\|_{L^2(D_h)^2}
\end{equation*}
with \[
C_1(L,\omega,h)=C (1+L^2)^{1/2}(\omega(h-m)+1).
\]
 Note $|\text{curl} u|^2=|\nabla u|^2-|\text{div}u|^2$, which completes the proof.
\end{proof}
Next it needs to estimate $\|\text{div}\,u\|_{L^2(D_h)}$ and $\|\text{curl}\,u\|_{L^2(\Gamma_h)}$. This is based on the \textit{a priori} bound for the Helmholtz equation in \cite{r7}. Set $H=h+1$ and extend the problem to $D_H$. Still denote the zero extension of $g$ in $D_H$ by $g$. The function $u$ can be extended to $D_H$ by \eqref{eq2.6} and we still denote the extension by $u$. In fact, we do not estimate $\|\text{div}\,u\|_{L^2(D_h)}$ and $\|\text{curl}\,u\|_{L^2(\Gamma_h)}$ but estimate $\|\text{div}\,u\|_{L^2(D_H)}$ and $\|\text{curl}\,u\|_{L^2(\Gamma_H)}$. The reason lies in the proof of Lemma 3.4. Recalling the Helmholtz decomposition \eqref{eq2.1}-\eqref{eq2.3}, $\phi$ and $\psi$ defined by \eqref{eq2.2} can also be extended to $D_H$. They both satisfy the Helmholtz equations\begin{equation} \label{eq3.11}
    \Delta w+k^2 w=g_0, \quad \text{in}\quad D_H
\end{equation}
with \[
k=k_s, g_0=-i/\omega^2 \text{div} {g}\quad \text{in}\quad D_H\quad \text{for}\quad w=\phi
\] and\[
k=k_p, g_0=-i/\omega^2 \text{curl} {g}\quad \text{in}\quad D_H\quad \text{for}\quad w=\psi.
\]
And it is easy to check they both satisfy (see \cite{r7}) the UPRC for the Helmholtz equation
\begin{equation} \label{eq3.12}
    w=\frac{1}{\sqrt{2\pi}}\int_{\mathbb{R}}\exp{i\sqrt{k^2-\xi^2}(x_2-H)+ix_1\xi}\hat{w}(\xi,H)\,\text{d}\xi,\quad x_2>H.
\end{equation}
It implies that (see \cite{r7}) $w$ satisfies TBC\begin{equation}\label{eq3.13}
    \partial_n w= \tilde{\mathcal{T}} w, \quad \text{on} \quad \Gamma_H,
\end{equation}
where $\tilde{\mathcal{T}}$ is the DtN operator from $H^{1/2}(\Gamma_H)$ to $H^{-1/2}(\Gamma_H)$ defined by\[
\tilde{\mathcal{T}} v =\mathcal{F}^{-1}(i\sqrt{k^2-\xi^2}\hat{v}), \quad v\in H^{1/2}(\Gamma_H).
\]
By Lemma 3.3, $\|w\|_{L^2(S)}$ can be estimated for $w=\phi$ or $w=\psi$. Hence it suffices to estimate $\|w\|_{L^2(D_H)}$ by $\|g_0\|_{L^2(D_H)}$ and $\|w\|_{L^2(S)}$.
  To this end, we construct a Dirichlet boundary value problem for the Helmholtz equation with inhomogeneous term  to estimate  $\|\partial_n w\|_{L^2(S)}$ by $\|g_0\|_{L^2(D_H)}$ and $\|w\|_{L^2(S)}$ and use the second Green's formula to estimate $\|w\|_{L^2(D_H)}$ by $\|\partial_n w\|_{L^2(S)}$. The stability result for the Helmholtz equation in \cite{r7} is used in the proof.
\begin{Le}
    The function $ w \in H^1(D_H) $ is assumed to satisfy \eqref{eq3.11} and \eqref{eq3.13}. Then the inequality\begin{equation*}
        \|w\|_{L^2(\Gamma_H)} \le  \|w\|_{L^2(
        D_H)} \le \tilde{C}_2(L,k,h) \|w\|_{L^2(S)}+ \tilde{C}_3(k,h) \|g_0\|_{L^2(D_H)} 
    \end{equation*} holds with \[
\tilde{C}_2(L,k,h)=C(1+L^2)^{1/4}\sqrt{H-m}(1+k(H-m))
    \]and\[
 \tilde{C}_3(k,h)=C(H-m)(1+k(H-m))^2/k
    .\]
\end{Le}
\begin{proof}
     Consider the boundary value problem\begin{align}\label{eq3.14}
        \Delta v+ k^2 v= \bar{w}\quad \text{in} \quad D_H, \\ \label{eq3.15}
        v=0 \quad \text{on~~} S,\quad  \\ \label{eq3.16}
        \partial_n v= \tilde{\mathcal{T}} v \quad\text{on}\quad \Gamma_H.
    \end{align}
    By the Theorem 4.1 in \cite{r7}, the inequality
    \begin{equation} \label{eq3.17}
        \|\nabla v\|_{L^2(D_H)}+k\|v\|_{L^2(D_H)} \le C (1+k(H-m))^2(H-m) \|w\|_{L^2(D_H)}
    \end{equation} holds.
    Furthermore, the Rellich identity for the Helmholtz equation gives (see \cite{r7})\begin{equation}\label{eq3.18}
        2\Re \int_{D_H} \partial_2\bar{v} \bar{w} \,\text{d}x =\left(\int_{\Gamma_H}+\int_{S}\right)\{2\Re (\partial_n v \partial_2 \bar{v})-n_2|\nabla v|^2+n_2k^2|v|^2\}\,\text{d}s.
    \end{equation}
    Moreover, the Lemma 2.2 in \cite{r7} yields
    \begin{equation}\label{eq3.19}
        \int_{\Gamma_H} 2\Re (\partial_n v\partial_2 \bar{v})-n_2 |\nabla v|^2 +n_2k^2|v|^2 \,\text{d}s \le 2k \Im\int_{D_H}\bar{v} \bar{w}\,\text{d}x.
    \end{equation}
    By $w=0$ on $S$, we have $\partial_\tau w=0$ on $S$. It turns out that \begin{align}\label{eq3.20}
       -\int_{S} \{2\Re (\partial_n v \partial_2 \bar{v})-n_2|\nabla v|^2+n_2k^2|v|^2\}\,\text{d}s&=-\int_{S} n_2|\partial_n v|^2 ds\notag\\ &\ge (1+L^2)^{-1/2} \| \partial_n v\|^2_{L^2(S)}.
    \end{align}
    Combining \eqref{eq3.17}-\eqref{eq3.20}, the inequality\begin{align}\label{eq3.21}
        \|\partial_n v\|^2_{L^2(S)} &\le (1+L^2)^{1/2}\left(2k\Im \int_{D_H}\bar{v} \bar{w}\,\text{d}x- 2\Re \int_{D_H} \partial_2\bar{v} \bar{w} \,\text{d}x\right)\notag\\
        &\le2(1+L^2)^{1/2} \|w\|_{L^2(D_H)}(k\|v\|_{L^2(D_H)}+\|\nabla v\|_{L^2(D_H)})\notag \\ &\le C(1+L^2)^{1/2}(H-m)(1+k(H-m))^2\|w\|^2_{L^2(D_H)} 
    \end{align} holds.
    By the second Green's formula, we have \begin{equation}\label{eq3.22}
        \int_{D_H} w\Delta v-v \Delta w\,\text{d}x =\left(\int_{\Gamma_H}+\int_{S}\right)\{w \partial_n v- v\partial_n w \}\,\text{d}s.
    \end{equation}
    Similarly as \eqref{eq3.4}, the second Green's formula can not be directly applied because the domain $D_H$ is unbounded. Noting that \[v \in H^1_S(D_H)=\{u \in H^1(D_H): u=0 \,\text{on} \,S \,\text{in trace sense}\},\]
     we have a sequence $\{v_n\} \subset C^\infty_0(D_H \cup \Gamma_H)$ such that \[
v_n \to v,  \quad \text{in}\quad H^1_S(D_H).
    \]  Applying the second Green's formula to $w_n$ and $v_n$ and taking the limit give that the second Green's formula holds for $w$ and $v$.
   Combining equations \eqref{eq3.11}, \eqref{eq3.14}, boundary condition \eqref{eq3.13}, \eqref{eq3.15}-\eqref{eq3.16}, and \eqref{eq3.22} yields \begin{equation}\label{eq3.23}
\int_{D_H} |w|^2 \,\text{d} x=\int_{D_H}w(\Delta v+k^2 v) \,\text{d}x= \int_{D_H} v g_0 \,\text{d} x+ \int_{S} w \partial_n v \,\text{d}s.
    \end{equation}
    Combining \eqref{eq3.17}, \eqref{eq3.21} and \eqref{eq3.23} yields \begin{align*}
        \|w\|^2_{L^2(D_H)} \le & \|v\|_{L^2(D_H)}\| g_0\|_{L^2(D_H)}+\|w\|_{L^2(S)}\|\partial_n v\|_{L^2(S)} \notag \\  \le & C\sqrt{H-m}(1+L^2)^{1/4} (1+k(H-m))\|w\|_{L^2(D_H)}\|w\|_{L^2(S)}  \notag \\ &+ C(H-m)\frac{(1+k(H-m))^2}{k}\|w\|_{L^2(D_H)}\|g_0\|_{L^2(D_H)}.
    \end{align*}
    This completes the right inequality in Lemma 3.4.
    To estimate $\|w\|_{L^2(\Gamma_H)}$, we use the fact that the UPRC (3.12) holds for all $c \in (h,H]$ (see \cite{r7}), which implies that \begin{equation*}
        \|w\|_{L^2(\Gamma_H)} \le \|\hat{w}\|_{L^2(\Gamma_c)}=\|w\|_{L^2(\Gamma_c)}, \quad  \text{ for }\, h<c\le H.
    \end{equation*}
    Integration with respective to $x_2$ gives \begin{equation*} 
    (H-h)\|w\|^2_{L^2(\Gamma_H)} \le \|w\|^2_{L^2(D_H \backslash D_h)} \le \|w\|^2_{L^2(D_H)},
    \end{equation*} which
    completes the proof.
\end{proof}
Applying this lemma to $w=\phi$ and $\psi$ yields  \begin{align}\label{eq3.24}
    \| \text{div}\, u\|^2_{L^2(\Gamma_H)}+\| \text{curl}\, u\|^2_{L^2(\Gamma_H)} \le 
      \| \text{div}\, u\|^2_{L^2(D_H)}+\| \text{curl}\, u\|^2_{L^2(D_H)} \notag \\ \le C_2(\omega, h, L)^2 (\| \text{div}\, u\|^2_{L^2(S)}+\| \text{curl}\, u\|^2_{L^2(S)}) + C_3(\omega,h)^2 \|g\|^2_{H^1(D_h)},
\end{align}
where\[
 C_2(\omega, h, L)=C(1+L^2)^{1/4}\sqrt{H-m}(1+\omega (H-m))
\]
and \[
C_3(\omega ,h)=C(H-m)(1+\omega (H-m))^2/\omega.
\]
Together with \eqref{eq3.24} and Lemma 3.3 gives \begin{align} \label{eq3.25}
    &\| \text{div}\, u\|^2_{L^2(\Gamma_H)}+\| \text{curl}\, u\|^2_{L^2(\Gamma_H)} \notag \\ &\le  C_2(\omega, h, L)^2C_1(\omega, h, L)\|g\|_{H^1(D_h)}\|\partial_2 u\|_{L^2(D_H)} + C_3(\omega, h)^2 \|g\|^2_{H^1(D_h)}
\end{align}
and \begin{align}\label{eq3.26}
    &\| \text{div}\, u\|^2_{L^2(D_H)}+\| \text{curl}\, u\|^2_{L^2(D_H)} \notag \\ &\le  C_2(\omega, h, L)^2C_1(\omega, h, L)\|g\|_{H^1(D_h)}\|\partial_2 u\|_{L^2(D_H)} + C_3(\omega, h)^2 \|g\|^2_{H^1(D_h)}.
\end{align}
Now it proceeds to estimate $\|\nabla u\|_{L^2(D_h)}^2$ by another Relliich identity for Navier equations, which indicates the following \textit{a priori} bound.
\begin{Th} Suppose that $g \in H^1(D)^2$ and $u \in V_h$ is a solution to Variation problem 1. Then the inequality\begin{equation*}
    \|u\|_{H^1(D_h)} \le (h-m+2) (C_4(\omega, h)+C_5(\omega,h)+C_6(\omega,h,L))\|g\|_{H^1(D_h)}
    \end{equation*}
    holds with \[
C_4(\omega,h)= C(h+1-m)\omega,\quad
C_5=C\sqrt{1+\omega^{-1}}C_3(\omega,h)
    \]
    and\[
C_6=C(\omega^{-1}+1)C_1(\omega,h,L)
C_2(\omega,h,L)^2.
    \]
\end{Th}
\begin{proof}
Assume $f \in C^2(\mathbb{R})$. Multiplying the Navier equations by $(x_2-m) \partial_2\bar{u}$ and using integration by parts gives \begin{align}\label{eq3.27}
        2\Re &\int_{D_H} g\cdot (x_2-m) \partial_2\bar{u} \,\text{d}x \notag \\&=\int_{D_H}\mathcal{E}(u,\bar{u})-2 \Re \sum_{j=1}^{2} \mathcal{E}(u, (x_2-m)e_j) \partial_2 \bar{u_j} -\omega^2 |u|^2 \,\text{d}x\notag \\
        &+  \left(\int_{S}+\int_{\Gamma_H}\right)\{2\Re (\mathcal{T}u\cdot \partial_2\bar{u})-\mathcal{E}(u,\bar{u})+\omega^2|u|^2\}(x_2-m)\,\text{d}s.
    \end{align} 
    Taking $v=u$ in variation formula implies \[
\int_{D_H} \mathcal{E}(u,\bar{u})-\omega^2|u|^2\,\text{d}x-\int_{\mathbb{R}} M(\xi)\hat{u}(\xi,H) \cdot \bar{\hat{u}}(\xi,H) \,\text{d}s=-\int_{D_H} g\cdot \bar{u}\,\text{d}x.
   \]
   Taking the real part and using lemma 3.1 gives \begin{align}\label{eq3.28}
   \int_{D_H} \mathcal{E}(u,\bar{u})-\omega^2|u|^2\,\text{d}x= \Re\int_{\mathbb{R}} M(\xi)\hat{u}(\xi,H) \cdot \bar{\hat{u}}(\xi,H) \,\text{d}\xi -\Re \int_{D_H} g\cdot \bar{u}\,\text{d}x\notag \\ \le -\Re \int_{D_H} g\cdot \bar{u}\,\text{d}x+\Re\int_{|\xi|\le k_s} M(\xi)\hat{u}(\xi,H) \cdot \bar{\hat{u}}(\xi,H) \,\text{d}\xi.
   \end{align} 
   Recalling $u=0$ and $\partial_\tau u=0$ on $S$ means \begin{align}\label{eq3.29}
&\int_{S}2\{\Re (\mathcal{T}u\cdot \partial_2\bar{u})-\mathcal{E}(u,\bar{u})+\omega^2|u|^2\}(x_2-m)\,\text{d}s\notag \\&=\int_{S}n_2(x_2-m)(\mu |\partial_n u|^2+(\lambda+\mu) |\text{div} u|^2 ) \le 0.
   \end{align}
   Combining \eqref{eq3.27}-\eqref{eq3.29} gives
   \begin{align}\label{eq3.30}
      &\int_{D_H} 2 \Re \sum_{j=1}^{2} \mathcal{E}(u, (x_2-m)e_j) \partial_2 \bar{u}_j \,\text{d}x \notag \\ &\le \int_{D_H} -g \cdot u-2\Re(g \cdot \partial_2 \bar{u})(x_2-m) \,\text{d}x+\Re\int_{|\xi|\le k_s} M(\xi)\hat{u}(\xi,H) \cdot \bar{\hat{u}}(\xi,H) \,\text{d}\xi\notag \\ &+(H-m)\int_{\Gamma_H} 2\Re (\mathcal{T}u\cdot \partial_2\bar{u})-\mathcal{E}(u,\bar{u})+\omega^2|u|^2\,\text{d}s.
   \end{align}
   Consider the left term first. There exist (see \cite{r13}) constants $C_1,C_2>0$ both independent on $\omega,h$ and $L$ such that \begin{align}\label{eq3.31}
       \int_{D_H} 2 \Re \sum_{j=1}^{2} \mathcal{E}(u, (x_2-m)e_j) \partial_2 \bar{u}_j \,\text{d}x+C_1 (\|\text{div}\,u\|^2_{L^2(D_H)}+\|\text{curl}\,u\|^2_{L^2(D_H)})\notag \\ \ge C_2 \|\nabla u\|^2_{L^2(D_H)^2}.
   \end{align}
   Then we estimate the three parts of the right term in \eqref{eq3.30} respectively. It is easy to see \begin{align}\label{eq3.32}
       \int_{D_H} -g \cdot u-2\Re(g \cdot \partial_2 \bar{u})(x_2-m) \,\text{d}x &\le \|g\|_{L^2(D_H)^2}\|u\|_{L^2(D_H)^2}\notag \\ &+2(H-m)\|g\|_{L^2(D_H)^2}\|\partial_2 u\|_{L^2(D_H)^2}.
   \end{align}
   Inserting the Poincar$\acute{\rm e}$ inequality \eqref{eq3.9} into \eqref{eq3.32},
   \begin{align}\label{eq3.33}
     \int_{D_H} -g \cdot u-2\Re(g \cdot \partial_2 \bar{u})(x_2-m) \,\text{d}x \le  C (H-m)\|g\|_{L^2(D_H)^2}\|\partial_2 u\|_{L^2(D_H)^2}.
   \end{align}
   By Lemma 3.2 and the Poincar$\acute{\rm e}$ inequality \eqref{eq3.9}, \begin{align}
   \label{eq3.34}\int_{\Gamma_H} 2\Re (\mathcal{T}u\cdot \partial_2\bar{u})-\mathcal{E}(u,\bar{u})+\omega^2|u|^2\,\text{d}s &\le 2k_s \|g\|_{L^2(D_H)^2}\|u\|_{L^2(D_H)^2}\notag \\ &\le C(H-m)k_s\|g\|_{L^2(D_H)^2}\|\partial_2 u\|_{L^2(D_H)^2}.
   \end{align}
   So the only difficulty is to estimate the second part of the right term. By \eqref{eq2.4}, Lemma 3.1 and the Plancherel identity,
   \begin{align}\label{eq3.35}
       &\Re\int_{|\xi|\le k_s} M(\xi)\hat{u}(\xi,H) \cdot \bar{\hat{u}}(\xi,H) \,\text{d}\xi \le C \int_{|\xi|\le k_s} \|M\| |\hat{u}(\xi,H)|^2 \,\text{d}\xi \notag \\ &\le C\omega k_s^2\int_{|\xi|\le k_s} (|P_H|^2+|S_H|^2) \,\text{d} \xi = C \omega k_s^2(\|\phi\|^2_{L^2(\Gamma_H)}+\|\psi\|^2_{L^2(\Gamma_H)}).
   \end{align}
 By \eqref{eq2.2} and \eqref{eq3.25}, \begin{align}   \label{eq3.36}&\|\phi\|^2_{L^2(\Gamma_H)}+\|\psi\|^2_{L^2(\Gamma_H)} \notag \\ &\le \frac{1}{k_p^4}(C_2(\omega, h, L)^2C_1(\omega, h, L)\|g\|_{H^1(D_h)^2}\|\partial_2 u\|_{L^2(D_H)^2} + C_3(\omega, h)^2 \|g\|^2_{H^1(D_h)^2}).
 \end{align}
 Inserting \eqref{eq3.36} into \eqref{eq3.35} gives \begin{align}\label{eq3.37}
     &\Re\int_{|\xi|\le k_s} M(\xi)\hat{u}(\xi,H) \cdot \bar{\hat{u}}(\xi,H) \,\text{d}\xi \le C \omega^{-1} \notag \\ &\left(C_2(\omega, h, L)^2C_1(\omega, h, L)\|g\|_{H^1(D_h)^2}\|\partial_2 u\|_{L^2(D_H)^2} + C_3(\omega, h)^2 \|g\|^2_{H^1(D_h)^2}\right).
 \end{align}
 Combining \eqref{eq3.26}, \eqref{eq3.30}-\eqref{eq3.31}, \eqref{eq3.33}-\eqref{eq3.34} and \eqref{eq3.37} implies \begin{align*}
     &\|\nabla u\|^2_{L^2(D_H)^2} \le C(H-m)\omega\|g\|_{L^2(D_H)^2}\|\partial_2 u\|_{L^2(D_H)^2} + C (\omega^{-1}+1) \notag \\ & \left(C_2(\omega, h, L)^2C_1(\omega, h, L)\|g\|_{H^1(D_h)^2}\|\partial_2 u\|_{L^2(D_H)^2} + C_3(\omega, h)^2 \|g\|^2_{H^1(D_h)^2}\right).
 \end{align*}
 It is easy to see \begin{equation*}
     C_4(\omega,h)\|g\|_{L^2(D_H)^2}\|\partial_2 u\|_{L^2(D_H)^2} \le C_4^2(\omega,h)^2 \|g\|^2_{L^2(D_H)^2} +\frac{1}{4}\|\partial_2 u\|^2_{L^2(D_H)^2}
 \end{equation*}
 and \begin{equation*}
     C_6(\omega,h)\|g\|_{L^2(D_H)^2}\|\partial_2 u\|_{L^2(D_H)^2} \le C_6^2(\omega,h)^2 \|g\|^2_{L^2(D_H)^2} +\frac{1}{4}\|\partial_2 u\|^2_{L^2(D_H)^2}.
 \end{equation*}
 It turns out that \[
\|\nabla u\|^2_{L^2(D_H)^2} \le ( C_4(\omega,h)^2 + C_5(\omega,h)^2+C_6(\omega,h,L)^2) \|g\|^2_{H^1(D_h)^2}.
 \]
 Recalling Poincar$\acute{\rm e}$ inequality \eqref{eq3.9} gives \begin{align*}
   \|u\|_{H^1(D_h)^2} &\le \|u\|_{H^1(D_H)^2} \notag \\ & \le (h-m+2) ( C_4(\omega,h) + C_5(\omega,h)+C_6(\omega,h,L)) \|g\|_{H^1(D_h)^2}.
 \end{align*}
 
  The conclusion for $f \in C^2(\mathbb{R})$ has been proven. Notice that the coefficient of $\|g\|_{H^1(D_h)}$ in \eqref{eq3.33} is an increasing function with respect to $L$. So it is allowed to extend this \textit{a priori} bound to any Lipschitz continuous $f$ by the method of approximating in \cite{r13}. This completes the proof.
\end{proof}
In practice, the frequency $\omega$ is usually assumed to be large. Hence, it is easy to verify when $\omega$ is very large, the stability result can be simplified to \[
\|u\|_{H^1(D_h)} \le C \omega^3 \|g\|_{H^1(D_h)},
\]where $C$ is independent on $\omega$.

The stability result directly implies uniqueness. In fact, it also implies existence by semi-Fredholm operator theory. Note that existence and uniqueness do not require $g \in H^1(D)^2$.
\begin{Th}
    Suppose that $g \in L^2(D)^2$, Variation problem 1 admits a unique solution $u \in V_h$.
\end{Th}
The proof can be found in \cite{r13} and hence be omitted here.

\section{Well-posedness and an \textit{a priori} bound for random case}
In this section, we will consider the well-posedness of VP 2. The proof is based on the general framework by Pembery and Spence in \cite{r17}. 
Firstly we show both the sesquilinear form $\mathcal{B}$ and the antilinear functional $\mathcal{G}$ are well-defined which is based on measurability and $\mathbb{P}$-essentially separability of $c$.
For measurability and $\mathbb{P}$-essentially separability of $c$, the following condition is necessary.
\begin{Co}
    The map $c_1$: $\Omega \to \mathcal{C}_1$ defined by \[
c_1(\eta)=f(\eta)
    \] satisfies $c_1 \in L^2(\Omega; \mathcal{C}_1)$ and the map $c_2$: $\Omega \to \mathcal{C}_2$ defined by \[
c_2(\eta)=\tilde{g}(\eta)
    \] satisfies $c_2 \in L^2(\Omega; \mathcal{C}_2)$.
\end{Co}
It implies the following lemma.
\begin{Le}
    Under Condition 4.1, the map $c$ is measurable and $\mathbb{P}$-essentially separable.
\end{Le}
\begin{proof}
    Since Condition 4.1 means $c_1$ and $c_2$ are strongly measurable, by Pettis measurability theorem (see \cite{r27}) they are measurable and $\mathbb{P}$-essentially separable. So $c=c_1 \times c_2$ is measurable and $\mathbb{P}$-essentially separable (see \cite{r27}).
\end{proof}
Then prove that the sesquilinear form $\mathcal{B}$ is well-defined by the continuity of $\mathscr{B}$ and the regularity of $\mathscr{B}\circ c$.
\begin{Le}
   (i) The map $\mathscr{B}$: $\mathcal{C} \to B(V_h,V_h^*) $ is continuous. 
   
   (ii)The map $\mathscr{B}\circ c \in L^\infty(\Omega; B(V_h,V_h^*)).$

   (iii) The sesquilinear form $\mathcal{B}$ is well-defined on $L^2(\Omega ; V_h) \times L^2(\Omega ; V_h). $
\end{Le}
\begin{proof}
    (i) For convenience, we only prove the continuity at the point $(f_0,g_0)\in \mathcal{C}$. At the other points, the proof of continuity is similar. Consider the sequence $\{(f_m, g_m)\} \subset \mathcal{C}$ such that $(f_m, g_m) \to (f_0,g_0)$ in $\mathcal{C}$ when $m \to \infty$.
     Denote the transform by
\begin{equation*}
    \mathcal{H}_m(y)=y_2+\alpha(y_2-f_0(y_1))(f_m(y_1)-f_0(y_1))e_2,\quad y \in D_h.
\end{equation*}
    For any $u,v \in V_h$, \begin{align*}
B_{(f_m, g_m)}(u,v)&-B(u,v)= 
\mu\int_{D_h} \sum_{j=1}^2\nabla u_j(I_2-\mathcal{J}_{\mathcal{H}^{-1}_m}\mathcal{J}_{\mathcal{H}^{-1}_m}^\top\det{\mathcal{J}_{\mathcal{H}_m}})\nabla \bar{v}_j\,\text{d}x \\
    &+(\lambda+\mu)\int_{D_h} (\nabla\cdot u)(\nabla\cdot \bar{v})-(\nabla\tilde{u}:\mathcal{J}_{\mathcal{H}^{-1}_m})(\nabla\bar{\tilde{v}}:\mathcal{J}_{\mathcal{H}^{-1}_m}^\top)\det{\mathcal{J}_{\mathcal{H}_m}}\,\text{d}x \\
    &-\omega^2\int_{D_h}u\cdot\bar{v}(\det{\mathcal{J}_{\mathcal{H}_m}}-1)\,\text{d}x.
    \end{align*} 
    By direct calculation, we have\begin{equation} \label{eq4.1}
\det{\mathcal{J}_{\mathcal{H}_m}}=1+O(\|f_m-f_0\|_{1,\infty}), \quad \mathcal{J}_{\mathcal{H}^{-1}_m}=I_2+O(\|f_m-f_0\|_{1,\infty}),
    \end{equation} which implies that\begin{equation}\label{eq4.2}
\mathcal{J}_{\mathcal{H}^{-1}_m}\mathcal{J}_{\mathcal{H}^{-1}_m}^\top\det{\mathcal{J}_{\mathcal{H}_m}}=I_2+O(\|f_m-f_0\|_{1,\infty}).
    \end{equation} These conclusions show that\begin{equation}\label{eq4.3}
        |B_{(f_m,g_m)}(u,v)-B(u,v)|\le C\|u\|_{H^1(D_h)^2}\|v\|_{H^1(D_h)^2}\|f_m-f_0\|_{1,\infty}.
    \end{equation}
    It turns out when $m \to \infty$, \begin{equation*}
\|B_{(f_m,g_m)}-B\|_{B(V_h,V_h^*)}\le C \|f_m-f_0\|_{1,\infty} \to 0.
    \end{equation*} This completes the proof.
    
    (ii)
    For any $u,v \in V_h$, we have
\begin{equation*}
    |B_{c(\eta)}(u,v)|\le|B_{c(\eta)}(u,v)-B(u,v)|+|B(u,v)|.
\end{equation*}
    Similarly as \eqref{eq4.3}, we have
    \begin{equation*}
        |B_{c(\eta)}(u,v)|\le C\|u\|_{H^1(D_h)^2}\|v\|_{H^1(D_h)^2}(\|f(\eta)-f_0\|_{1,\infty}+1).
    \end{equation*}
    Recall \begin{equation*}
\|f(\eta)-f_0\|_{1,\infty} \le M_0.
    \end{equation*} 
    It implies that\[
|B_{c(\eta)}(u,v)|\le C(M_0+1)\|u\|_{H^1(D_h)^2}\|v\|_{H^1(D_h)^2} \infty,
    \] which means $\mathscr{B}\circ c \in L^\infty(\Omega; B(V_h,V_h^*)).$

    (iii) In order to show $\mathcal{B}$ is well-defined, we must show $B_{c(\eta)}(v_1,v_2)$ is integrable for any $v_1, v_2 \in L^2(\Omega; V_h)$ and $B_{c(\eta)}(v_1,\cdot) \in L^2(\Omega; V_h)$. Combining (i), (ii) and applying Lemma 2.7 in \cite{r17} complete this proof.
\end{proof}
Next give a similar lemma for the antilinear functional $\mathcal{G}$.
\begin{Le}
(i)
The map $\mathscr{G}$: $\mathcal{C}\to V_h^*$ is continuous.

(ii) The map $\mathscr{G}\circ c \in L^2(\Omega; V_h^*).$

(iii) The antilinear functional $\mathcal{B}$ is well-defined on $L^2(\Omega ; V_h). $
\end{Le}
\begin{proof}
    (i) Similarly as Lemma 4.2, we assume $(f_m,g_m) \to (f_0,g_0)$ in $\mathcal{C}$. For any $v \in V_h$,
    \[
G_{(f_m,g_m)}(v)-G(v)=\int_{D_h} (g_0-g_m\det{\mathcal{J}_{\mathcal{H}_m}})\cdot\bar{v}\,\text{d}x.
    \]
    So we have\begin{align*}
|G_{(f_m,g_m)}(v)-G(v)|&\le \int_{D_h}|g_0-g_m||v|\,\text{d}x+\int_{D_h}|g_m\det{\mathcal{J}_{\mathcal{H}_m}}-g_m||v|\,\text{d}x\\
&\le C\|v\|_{H^1(D_h)^2}(\|g_0-g_m\|_{L^2(D_h)^2}+\|g_m\|_{L^2(D_h)^2}\|f_m-f_0\|_{1,\infty}).
    \end{align*}
     It turns out when $m \to \infty$, \begin{equation}\label{eq4.4}
\|G_{(f_m,g_m)}-G\|_{V_h^*}\le C(\|g_0-g_m\|_{L^2(D_h)}+\|f_m-f_0\|_{1,\infty})\to 0.
    \end{equation} So $\mathscr{G}$ is continuous.
    
    (ii) For any $v \in V_h$, we have
    \begin{equation*}
        |G_{c(\eta)}(v)|\le |G_{c(\eta)}(v)-G(v)|+|G(v)|.
    \end{equation*}
    Similarly as \eqref{eq4.4}, we can see \begin{equation*}
        \|G_{c(\eta)}\|_{V_h^*}\le C(\|g_0-\tilde{g}(\eta)\|_{H^1(D_h)^2}+M_0+1).
    \end{equation*}
    Since probability measure is finite, \[
C(M_0+1) \in L^2(\Omega).
    \]
    Condition 4.1 means \[
C(\|g_0-\tilde{g}(\eta)\|_{H^1(D_h)^2})  \in L^2(\Omega).
    \] So we have $\|G_{c(\eta)}\|_{V_h^*} \in L^2(\Omega)$ which completes the proof.
    
    (iii) In order to show $\mathcal{G}$ is well-defined, we must show $G_{c(\eta)}(v_1)$ is integrable for any $v_1 \in L^2(\Omega; V_h)$ and $G_{c(\eta)} \in L^2(\Omega; V_h)$. Combining (i), (ii) and applying Lemma 2.7 in \cite{r17} completes this proof.
\end{proof}

For any given $\eta$, we consider the  following deterministic variation problem.

\emph{Variation problem 3} (VP 3) Find $u(\eta) \in V_h$ such that $B_{c(\eta)}(u(\eta),v)=G_{c(\eta)}(v),\quad \forall v \in V_h.$

The existence and uniqueness of VP 3 is directly deduced by Theorem 3.2. The \textit{a priori} bound in Theorem 3.1 can also be used for VP 3. Notice that $L(\eta)$ satisfies \[
L(\eta) \le L+M_0.
\] 
\begin{Th}
For any given $\eta$, Variation problem 3 admits a unique solution $u(\eta) \in V_h$. 
And the \textit{a priori} bound\[\|u^*(\eta)\|_{H^1(
D_h(\eta))^2}
\le (h-m+2)(C_4(\omega, h )+C_5(\omega, h )+C_6(\omega, h,L_0) ) \|g(\eta)\|_{H^1(D_h(\eta))^2}
\] holds for $u^*(\eta)=u(\eta)\circ \mathcal{H}^{-1}$ with $L_0=M_0+L$.
\end{Th}
\begin{proof}
    For any given $\eta$, if $u(\eta)$ is a solution to VP 3, $u^*(\eta)=u(\eta)\circ \mathcal{H}^{-1}$ is solution to VP 1 corresponding to $f(\eta)$ and $g(\eta)$. Conversely, if $u(\eta)$ is solution to VP 1 corresponding to $f(\eta)$ and $g(\eta)$, $\tilde{u}(\eta)=u(\eta)\circ \mathcal{H}$ is solution to VP 3. So Theorem 3.2 implies existence and uniqueness of VP 3 and Theorem 3.1 implies the \textit{a priori} bound.
\end{proof}
Theorem 4.1 shows there exists a solution $u(\eta)$ to VP 3 for given $\eta$. In fact, we can prove $u(\eta) \in L^2(\Omega; V_h)$.
\begin{Le}
     For the solution $u(\eta)$ to Variation problem 3, $u(\eta) \in L^2(\Omega; V_h)$.
\end{Le}
\begin{proof}
By Bochner's integrability theorem (see \cite{r27}) we must prove $u(\eta)$ is strongly measurable and $\|u(\eta)\|^2_{H^1(D_h)^2} \in L^1(\Omega)$.

For $u(\eta)$ which is the solution to VP 3, taking $x=\mathcal{H}^{-1}(y)$ gives
\begin{align}\label{eq4.5}
     \|u(\eta)\|^2_{H^1(D_h)^2}&= \int_{D_h(\eta)} \sum_{j=1}^2\nabla{u}^*_j\mathcal{J}_{\mathcal{H}}\mathcal{J}_{\mathcal{H}}^\top\nabla \bar{u}^*_j\det{\mathcal{J}_{\mathcal{H}^{-1}}}\,\text{d}y\notag \\ &+\int_{D_h(\eta)} |u^*|^2\det{\mathcal{J}_{\mathcal{H}^{-1}}}\,\text{d}y
     \end{align} 
     with \[
u^*=u(\eta) \circ \mathcal{H}^{-1}.
     \]Recalling Section 2.2 and direct calculation give  \begin{equation}\label{eq4.6}
\mathcal{J}_{\mathcal{H}}=\left(\begin{array}{cc}
              1  & 0 \\
             J_1 & 1+J_2
         \end{array}\right)
     \end{equation}
     and \begin{equation}\label{eq4.7}
        \det{\mathcal{J}_{\mathcal{H}^{-1}}}=(1+J_2)^{-1}.
     \end{equation}
     Recalling \eqref{eq2.13} implies \begin{equation}\label{eq4.8}
|J_2|<\frac{M-m}{\gamma-2\delta}<1,\quad  |J_1| \le C M_0.
     \end{equation} Inserting \eqref{eq4.8} to \eqref{eq4.6}-\eqref{eq4.7} yields \begin{equation}\label{eq4.9}
      \det{\mathcal{J}_{\mathcal{H}}} \le C,\quad  \|\mathcal{J}_{\mathcal{H}}\| \le C.
     \end{equation}
 Then by \eqref{eq4.5}, \eqref{eq4.9} and Theorem 4.1, \begin{equation} \label{eq4.10}
     \|u(\eta)\|^2_{H^1(D_h)^2} \le C \|u^*\|^2_{H^1(D_h(\eta))^2} \le C\|g(\eta)\|^2_{H^1(D_h(\eta))^2} .
 \end{equation}
 Taking $x=\mathcal{H}(y)$ gives \begin{equation} \label{eq4.11}
     \|g(\eta)\|^2_{H^1(D_h(\eta))^2} \le \|\tilde{g}(\eta)\|^2_{H^1(D_h)^2} 
 \end{equation} similarly as \eqref{eq4.5}-\eqref{eq4.9}.
   Combining \eqref{eq4.10}-\eqref{eq4.11} gives \begin{equation*}
      \|u(\eta)\|^2_{H^1(D_h)^2} \le C \|\tilde{g}(\eta)\|^2_{H^1(D_h)^2}.
      \end{equation*}
      By Condition 4.1,  \[
 \|\tilde{g}(\eta)\|^2_{H^1(D_h)^2} \in L^1(\Omega),
      \]which shows $\|u(\eta)\|^2_{H^1(D_h)^2} \in L^1(\Omega)$.
      
      Next show $u(\eta)$ is strongly measurable. Define the solution operator $\mathcal{U} :\mathcal{C} \to V_h$ by \begin{equation*}
          \mathcal{U}(\phi, \psi)=u_{(\psi,\phi)} \quad \text{ for } (\psi,\phi) \in \mathcal{C},
      \end{equation*} where $u_{(\psi,\phi)}$ is the solution to VP 3 corresponding to $\phi, \psi$. Then we prove the solution operator is continuous.

Assume the sequence $\{(f_m,g_m)\} \subset \mathcal{C}$ satisfying $(f_m,g_m) \to (f_0,g_0)$ in $\mathcal{C}$. The variation formula \[
B_{(f_m,g_m)}(u_m,v)=G_{(f_m,g_m)}(v),\quad B(u,v)=G(v) \quad \forall v \in V_h
\] implies $u_m=B_{(f_m,g_m)}^{-1}G_{(f_m,g_m)}$ and $u=B^{-1}G$.
Then we obtain the inequality \begin{equation*}
    \|u_m-u\|_{H^1(D_h)^2} \le \|B_{(f_m,g_m)}^{-1}-B^{-1}\|\|G_{(f_m,g_m)}\|_{V_h^*}+\|B\|\|G_{(f_m,g_m)}-G\|.
\end{equation*}
Recalling \eqref{eq4.3} and \eqref{eq4.4} implies \[
\|u_m-u\|_{H^1(D_h)^2} \to 0,
\] which means the operator $\mathcal{U}$ is continuous.

      By the continuity of $\mathcal{U}$ and the strong measurability of $c$, we obtain (see \cite{r27}) $u(\eta)$ is strongly measurable which completes the proof.
\end{proof}
Based on the above conclusions, now we can prove the well-posedness of VP 2.  
 \begin{Th}
  Variation problem 2 exists a unique solution $u\in L^2(\Omega, V_h)$.
 \end{Th}
 \begin{proof}
      Together with Theorem 4.1 and Lemma 4.4 implies there exists a unique solution $u(\eta)$ to VP 3 for any $\eta \in \Omega$ and $u(\eta) \in L^2(\Omega; V_h)$. Combining Lemma 4.1, Lemma 4.3 and Theorem 2.8 in \cite{r17} shows this $u(\eta)$ is a solution to VP 2. Conversely, any solution to VP 2 is also the solution to VP 3 for a.s. $\eta$ (see Theorem 2.9 in \cite{r17}). So uniqueness of VP 3 directly implies uniqueness of VP 2.
 \end{proof}
 Then we can direct integrate the inequality in Theorem 4.1 with respect to $\eta$ and apply \eqref{eq4.10}-\eqref{eq4.11} to get the \textit{a priori} bound given by the following theorem.
 \begin{Th}
     Assume $u\in V_h(\eta)$ is the solution to Variation problem 1 corresponding to $f(\eta)$ and $g(\eta)$ for given $\eta \in \Omega$ which means $\tilde{u}(\eta)\in L^2(\Omega;V_h)$ is the solution to Variation problem 2. They respectively satisfy the bound  \begin{align*}
&\int_\Omega\|u\|^2_{H^1(D_h(\eta))^2}\mathrm{d}\,\mathbb{P}\\ &\le (H-m+1)^2(C_4(\omega, h )+(C_5(\omega, h )+(C_6(\omega, h,L_0) )^2\int_{\Omega}\|g\|^2_{H^1(D_h(\eta))^2}\mathrm{d}\,\mathbb{P},
     \end{align*} and
     \begin{align*} &\int_\Omega\|\tilde{u}\|^2_{H^1(D_h)^2}\mathrm{d}\,\mathbb{P}\\ &\le (H-m+1)^2(C_4(\omega, h )+(C_5(\omega, h )+(C_6(\omega, h,L_0) )^2\int_{\Omega}\|\tilde{g}\|^2_{H^1(D_h)^2}\mathrm{d}\,\mathbb{P}.
     \end{align*}
 \end{Th}
\section{Conclusion}
This paper establishes the well-posedness of deterministic and random elastic scattering from unbounded rough surface. An \textit{a priori} bound explicitly with frequencies is given for deterministic case and extended to random case. Future work will focus on elastic scattering with an incident plane wave, which is still remained unsolved since the Rellich identity is not valid any more and additional difficulties arises in this case. 

 \bibliographystyle{elsarticle-num} 
 \bibliography{cas-refs}





\end{document}